\documentclass[12pt,a4paper]{article}
\usepackage{amsmath}
\usepackage{amsthm}
\usepackage{amssymb}
\usepackage{graphicx}
\usepackage{tabularx}
\usepackage{color}
\usepackage{setspace}
\usepackage{capt-of}
\usepackage{float}
\usepackage{cite}
\usepackage{sectsty}
\usepackage{listings}
\usepackage{fancyhdr}
\usepackage[gen]{eurosym}
\usepackage{multirow}
\usepackage{textcomp}
\usepackage{fancyvrb}
\usepackage{subfig}
\usepackage{paralist}
\usepackage{epic}
\usepackage{mathtools}
\usepackage[mathscr]{euscript}
\usepackage{mathrsfs}
\newtheorem{definition}{Definition}[section]

\newtheorem{thm}[definition]{Theorem}

\newtheorem{cor}[definition]{Corollary}

\usepackage[top=2.5cm,bottom=2.5cm,left=2cm,right=2cm]{geometry}
\usepackage{multicol}
\usepackage{hyperref}
\usepackage{setspace}
\hypersetup{
	colorlinks,
	citecolor=blue,
	filecolor=black,
	linkcolor=black,
	urlcolor=black
}
\date{}
\begin{document}
	
	\begin{center}
		\textbf{Hilbert Space of Complex-Valued Harmonic Functions in the Unit Disc}
	\end{center}
	\vspace{2mm}
	\begin{center}
		Tseganesh Getachew Gebrehana and Hunduma Legesse Geleta\\\textbf{tseganesh.getachew@aau.edu.et and hunduma.legesse@aau.edu.et}\\
		\vspace{8mm}
		\textbf{Department of Mathematics, College of Natural and Computational Sciences, Addis Ababa University}
	\end{center}	     
	\vspace{4mm}
	First author: Tseganesh Getachew Gebrehana \\
	Corresponding author: Hunduma Legesse Geleta \\  
	\vspace{2mm}
	\pagenumbering{arabic}
	
	\vspace{3mm}       
	\textit{\textbf{Abstract}. We investigate an extended version of Hilbert space of analytic functions called Hilbert space of complex-valued harmonic functions. It is found that functions in Hilbert space of complex-valued harmonic functions exhibit many properties analogous to its analytic counter part such as complex-valued harmonic function analogous of norm, equivalent norms, reproducing kernels, growth estimates and Littlewood-Paley Identity Theorem.  In conclusion we prove that  many results in Hilbert space of analytic functions also hold in larger Hilbert space of complex-valued harmonic functions.} 
	\newline \\
	\textbf{Keywords/phrases}: Complex-valued harmonic functions; Hilbert space; Inner product; Norm; Integral means; Growth estimates; Reproducing kernel. 
	\section{Introduction}
	\pagenumbering{arabic}
	
	The family of complex-valued harmonic functions $f=u+iv$ defined in the unit disk $\mathbb{D}= \lbrace z: |z|<1 \rbrace$, where $u$ and $v$ are real harmonic in $\mathbb{D}$ ~($u$ and $v$ are not necessarily harmonic conjugates), were introduced by Clunie and SheilSmall \cite{CJTS} in 1984. In any simply connected domain $\mathnormal{G} \subset \mathbb{C}$, any such harmonic functions can be decomposed as $f = h + \overline{g}$, where $h$ and $g$ are analytic \cite{DPL1}. This family of complex-valued harmonic functions is a generalization of analytic mappings studied in geometric function theory, and much research has been done investigating the properties of these harmonic functions. For more details of the topic, see also  Duren \cite{DPL4}  and Dorff and Rolf \cite{DMJS} .
	\newline \\  
	The Hardy  space $\mathnormal{H^2} (\mathbb{D}) $ is defined to be the space of all analytic functions on the open unit disk $\mathbb{D}=\{z\in \mathbb{C}: \vert z \rvert < 1\}$ whose power series
	coefficients are square-summable. This space is equipped with norms which is derived from the inner product as in \cite{HBR}. Another equivalent norm has been also defined using integral means in this space. 
	By definition the sequence $\lbrace a_{n} \rbrace_{n=0}^{\infty}$ of power series coefficients belongs to the Hilbert space $\mathnormal{l^2}(\mathbb{Z^+}),$ and conversely, every sequence in $\mathnormal{l^2}(\mathbb{Z^+})$ defines an analytic function on the open unit disk belonging to  $\mathnormal{H^2}(\mathbb{D})$  by means of the map $\lbrace a_{n} \rbrace _{n=0}^{\infty} \rightarrow \sum_{n=0}^{\infty} a_{n}z^{n},$ 
	showing that $\mathnormal{H^2}(\mathbb{D})$ is isometrically isomorphic to $\mathnormal{l^2}(\mathbb{Z^+})$. From which it has been  concluded that the Hardy space $\mathnormal{H^2}(\mathbb{D})$ is a Hilbert space.
	Motivated by the results in Hilbert space of analytic function in the unit disc and methods used in finding norm equivalence, growth estimate and kernel on Hilbert space of analytic function, we extend these works on to Hilbert space of complex-valued harmonic functions.
	\newline \\
	Therefore, the purpose of this paper is to  explore different equivalent norms, find reproducing kernels, growth estimates,   and study operators associated  to an extended version of space of analytic functions called Hilbert space of complex-valued harmonic functions.  The Hilbert space of complex-valued harmonic functions is shown to exhibit a lot of  properties analogous to its counter part Hilbert space of analytic functions, including the Littlewood-Paley identity theorem. This paper is organized as follows: In section 2 we formally define Hilbert space of complex-valued harmonic functions. In Section 3, we study equivalent norm in terms of integral mean, growth estimates and kernel on space of complex valued harmonic functions in the unit disc.

	\section{Norm on $\mathnormal{H}_{h}^{2}(\mathbb{D})$}	
	In this section, we define a norm on space of complex-valued harmonic functions whose coefficients in the Taylor series representation is square summable and then show such space is a Hilbert space.
	\begin{thm}
		Let $f$ be a complex-valued harmonic function on the unit disc $\mathbb{D}$ given by 
		\begin{center}
			$f(z)=h(z)+\overline{g(z)}  $  
		\end{center}
		where $h(z)=\sum_{n=0}^{\infty} a_{n}z^{n}$ and $g(z)=\sum_{n=0}^{\infty} {b_{n}z^{n}}$ are analytic. Suppose \begin{center}
			$\mathnormal{H}_{h}^{2}(\mathbb{D})= \lbrace f:\mathbb{D} \rightarrow \mathbb{C}: f(z)= \sum_{n=0}^{\infty} a_{n}z^{n} + \sum_{n=0}^{\infty} \overline{b_{n}z^{n}}$ with $\sum_{n=0}^{\infty} |a_{n}|^{2}+|b_{n}|^{2} <\infty \rbrace$.
		\end{center} 	
		Then $||.||_{\mathnormal{H}_{h}^{2}}: {\mathnormal{H}_{h}^{2}}(\mathbb{D}) \rightarrow \mathbb{R}$  defined by $||f||_{\mathnormal{H}_{h}^{2}}=({\sum_{n=0}^{\infty}|a_{n}|^{2} + |b_{n}|^{2}})^{\frac{1}{2}}$ is a norm. 
	\end{thm}
	\begin{proof}
		\begin{itemize}
			\item [(i).] By definition, $||f||_{{H}_{h}^{2}}^{2} \geq 0$.
			\item [(ii).] $||\alpha f||_{{H}_{h}^{2}}^{2}= \sum_{n=0}^{\infty}|\alpha a_{n}|^{2} + |\alpha b_{n}|^{2}=|\alpha|^{2} \sum_{n=0}^{\infty} |a_{n}|^{2} + |b_{n}|^{2}= |\alpha|^{2}||f||_{{H}_{h}^{2}}^{2}$. \newline
			Thus, $||\alpha f||_{{H}_{h}^{2}}=|\alpha|||f||_{{H}_{h}^{2}}$.
			\item [(iii.)] $||f+F||_{{H}_{h}^{2}}^{2} = \sum_{n=0}^{\infty} |a_{n}+A_{n}|^{2} + |b_{n}+B_{n}|^{2}$ 
			\begin{center}
				$= \sum_{n=0}^{\infty}  |a_{n}|^{2}+ |A_{n}|^{2} + 2\Re (a_{n}\overline{A_{n}}) +|b_{n}|^{2}+|B_{n}|^{2}+ 2\Re (b_{n}\overline{B}_{n})$	
			\end{center} 
			\begin{center}
				$\leq ||f||_{{H}_{h}^{2}}^{2} + ||F||_{{H}_{h}^{2}}^{2}+2||f||_{{H}_{h}^{2}}||F||_{{H}_{h}^{2}}=(||f||_{{H}_{h}^{2}}+||F||_{{H}_{h}^{2}})^{2}$.
			\end{center}
			This implies that
			\begin{center}
				$||f+F||_{{H}_{h}^{2}} \leq ||f||_{{H}_{h}^{2}}+||F||_{{H}_{h}^{2}}.$
			\end{center} 
		\end{itemize}
		Therefore, 
		\begin{center}
			$||f||_{{H}_{h}^{2}}(\mathbb{D})$ is a norm on $\mathnormal{H}_{h}^{2}(\mathbb{D})$.
		\end{center} 
	\end{proof}
	This norm is induced from the inner product defined on ${\mathnormal{H}_{h}^{2}}(\mathbb{D})$ as follows,
	
	\begin{center}
		$\langle f , F \rangle= \sum_{n=0}^{\infty} (a_{n}\overline{A_{n}}+  \overline{b_{n}}B_{n}). $
	\end{center}
	where $f(z)=\sum_{n=0}^{\infty} a_{n}z^{n}+\sum_{n=0}^{\infty} \overline{b_{n}z^{n}}$  and $F(z)=\sum_{n=0}^{\infty} A_{n}z^{n}+\sum_{n=0}^{\infty} \overline{B_{n}z^{n}}$ are in $\mathnormal{H}_{h}^{2}(\mathbb{D}).$ 
	
	\begin{thm}
		Let $\mathnormal{H}_{h}^{2}$ be the space of complex-valued harmonic functions on the unit disc  $\mathbb{D}:=\lbrace z \in \mathbb{C} : |z|<1 \rbrace$ whose Taylor series representation has square-summable coefficients. Then $\mathnormal{H}_{h}^{2}$ is a Hilbert space.
	\end{thm}
	
	\begin{proof}
		Since the sequences $\lbrace a_{n} \rbrace_{n=0}^{\infty}$ and $\lbrace b_{n} \rbrace_{n=0}^{\infty}$ of power series coefficients belong to the Hilbert space $\mathnormal{l^2}(\mathbb{Z^+})$, then the sequence $\lbrace c_{n} \rbrace_{n=0}^{\infty}$ defined by $c_{2n} = a_n$ for each n= 0, 1, 2, ... and $c_{2n-1}= b_{n-1}$ for each n =1, 2, 3, ... belongs to the Hilbert space $\mathnormal{l^2}(\mathbb{Z^+})$. Conversely, every sequence in $\mathnormal{l^2}(\mathbb{Z^+})$ defines a complex-valued harmonic function on the open unit disk
		belongs to $\mathnormal{H_h^2}(\mathbb{D})$ by means of the map $\lbrace c_{n} \rbrace _{n=0}^{\infty} \rightarrow \sum_{n=0}^{\infty} a_{n}z^{n}+\sum_{n=0}^{\infty} \overline{b_{n}z^{n}}$. Thus, $\mathnormal{H_h^2}(\mathbb{D})$ is isometrically isomorphic to $\mathnormal{l^2}(\mathbb{Z^+})$. Thus we conclude that $\mathnormal{H_h^2}(\mathbb{D})$ is a Hilbert space.	
	\end{proof}

	\section{Equivalent Norms }
	We now consider Complex-valued harmonic function analogous  of equivalent norm by integral means and Littlewood-Paly Identity Theorem which were studied in space of analytic functions. 
	\begin{thm}
		The norm defined on $\mathnormal{H}_{h}^{2}$ (in Theorem 2.1) has another equivalent representation by integral means which is denoted by $\mathnormal{M}_{2}^{2} (f,r)$ and defined by 
		\begin{center}
			$\mathnormal{M}_{2}^{2} (f,r) = \frac{1}{2\pi} \int_{-\pi}^{\pi} |f(re^{i\theta})|^{2} d\theta$,
		\end{center}
		where $f(z)=\sum_{n=0}^{\infty} a_{n}{z}^{n}+\sum_{n=0}^{\infty} \overline{ b_{n}{z}^{n}}$ on $\mathbb{D}$ and $0 \leq r < 1$. 	
	\end{thm} 
	\begin{proof}
		Now using polar form of the series representation of $f$, we get
		\begin{center}
			$ f(re^{i\theta} )=\sum_{n=0}^{\infty} a_{n}{r}^{n}{e^{in\theta}}+\sum_{n=0}^{\infty} \overline{ b_{n}{r}^{n}{e^{in\theta}}}$.
		\end{center}
		After some algebraic manipulation, we obtain 
		\begin{center}
			$ |f(re^{i\theta})|^{2}=\sum_{n=0}^{\infty} \sum_{m=0}^{\infty} a_{n}\overline{a_{m}}{r}^{n+m}{e^{i(n-m)\theta}} +\sum_{n=0}^{\infty}\sum_{m=0}^{\infty} a_{m}b_{n}{r}^{m+n}{e^{i(n+m)\theta}} + \sum_{n=0}^{\infty}\sum_{m=0}^{\infty}\overline{ a_{m}b_{n}{r}^{m+n}{e^{i(n+m)\theta}}}+\sum_{n=0}^{\infty} \sum_{m=0}^{\infty} b_{m}\overline{b_{n}}{r}^{n+m}{e^{i(n-m)\theta}}$. 
		\end{center}
		It is clear that the integral of exponential function $\lbrace e^{in\theta} \rbrace_{n=0}^{\infty}$ is $2\pi$ when $n=m$ and $0$ when $n \neq m$. Multiplying both sides of the above equation by $\frac{1}{2\pi}$ and integrating with respect to $\theta$ from $-\pi $ to $\pi$, we get 
		\begin{center}
			$ \frac{1}{2\pi}\int_{-\pi}^{\pi}|f(re^{i\theta})|^{2} d\theta=\sum_{n=0}^{\infty} |a_{n}|^{2}{r}^{2n} +\sum_{n=0}^{\infty} |b_{n}|^{2}{r}^{2n} +\frac{1}{2\pi}\int_{-\pi}^{\pi} \sum_{n=0}^{\infty}\sum_{n=0}^{\infty} 2\Re ({a_{n}b_{n}{r}^{2n}{e^{2in\theta}}})d\theta $
		\end{center}
		\begin{center}
			$=\sum_{n=0}^{\infty} (|a_{n}|^{2} + |b_{n}|^{2}){r}^{2n} + \sum_{n=0}^{\infty}\sum_{n=0}^{\infty} \frac{2{r}^{2n}}{2\pi} \Re ({a_{n}b_{n}\int_{-\pi}^{\pi} {e^{2in\theta}}}d\theta)$
		\end{center}
		\begin{center}
			$=\sum_{n=0}^{\infty} (|a_{n}|^{2} + |b_{n}|^{2}){r}^{2n}$ \hspace{6mm} 
		\end{center} 
		Therefore,
		\begin{center}
			$M_{2}^{2} (f,r)=\frac{1}{2\pi}\int_{-\pi}^{\pi}|f(re^{i\theta})|^{2}d\theta=\sum_{n=0}^{\infty} (|a_{n}|^{2} + |b_{n}|^{2}){r}^{2n}. $
		\end{center} 	
		To complete the proof we need to show	
		$||f||_{\mathnormal{H}_{h}^2}= \lim_{r \rightarrow 1^{-}} M_{2} (f,r)$.
		So, from the above equation we have 
		\begin{center}
			$M_{2}^{2}(f,r)=\sum_{n=0}^{\infty} (|a_{n}|^{2} + |b_{n}|^{2}){r}^{2n}$
			$ \leq \sum_{n=0}^{\infty} (|a_{n}|^{2} + |b_{n}|^{2})=||f||_{\mathnormal{H}_{h}^{2}}^{2}$,
		\end{center}
		whenever $f \in {\mathnormal{H}_{h}^{2}}$ and $0 \leq r<1$.
		So $M_{2} (f,r)$ is bounded by the $\mathnormal{H}_{h}^{2}$-norm. It remains to show that whenever $\lim_{r \rightarrow 1^{-}}M_{2} (f,r)=\mathnormal{M}< \infty$, then the partial sum of the series of the form $M_{2}^{2} (f,r)=\sum_{n=0}^{\infty} (|a_{n}|^{2} + |b_{n}|^{2}){r}^{2n} $ are bounded by $\mathnormal{M^2}:$
		\begin{center}
			$\sum_{n=0}^{\mathnormal{N}} (|a_{n}|^{2} + |b_{n}|^{2}){r}^{2n} \leq \sum_{n=0}^{\infty} (|a_{n}|^{2} + |b_{n}|^{2}){r}^{2n} \leq \mathnormal{M^2}$.
		\end{center} 
		As $r \rightarrow 1^{-}$, these partial sums converges to functions in $\mathnormal{H}_{h}^{2}$, which must therefore be bounded by $\mathnormal{M^2}$ as well. If every partial sum of Taylor series representation of functions in $\mathnormal{H}_{h}^{2}$ is bounded by  $\mathnormal{M^2}$ , then this is also true for the entire series. This completes the proof. 
	\end{proof}
	\begin{cor}
		The space of bounded complex-valued harmonic functions on the unit disc $\mathnormal{H}_{h}^{\infty}$ is a subset of $\mathnormal{H}_{h}^2$.
	\end{cor}	
	\begin{proof}
		Let $f \in \mathnormal{H}_{h}^\infty(\mathbb{D})$. Then $||f||_{\mathnormal{H}_{h}^\infty} = \sup_{z \in \mathbb{D}}|f(z)|$.  Now, $\frac{1}{2\pi}\int_{-\pi}^{\pi}|f(re^{i\theta})|^{2}d\theta \leq \frac{1}{2\pi}\int_{-\pi}^{\pi} \sup |f(re^{i\theta})|^{2}d\theta \newline =\frac{1}{2\pi}\int_{-\pi}^{\pi}||f||_{\mathnormal{H}_{h}^{\infty}}^{2}d\theta=||f||_{\mathnormal{H}_{h}^{\infty}}^{2}$,
		which holds true for every $0<r<1$.
		So for any $f \in {\mathnormal{H}_{h}^{\infty}} $  we get 
		\begin{center}
			$\lim_{r \rightarrow 1^{-}}M_{2} (f,r) \leq ||f||_{\mathnormal{H}_{h}^{\infty}}^{2}.$
		\end{center}
		Hence,  $f \in \mathnormal{H}_{h}^2$.
	\end{proof}	
	The following theorem is Littlewood-Paley identity theorem for space of complex-valued harmonic functions. It provides another expression for the $\mathnormal{H}_{h}^2$-norm. 
	\begin{thm}
		For every complex-valued harmonic functions $f \in \mathnormal{H}_{h}^2$ on the unit disc we have 
		\begin{center}
			$||f||_{\mathnormal{H}_{h}^2}=|f(0)|^{2}+2\int_{\mathbb{D}}|f'(z)|^{2} \log \frac{1}{|z|} dA(z), $
		\end{center}
		where $d\mathnormal{A}$ denotes the normalized Lebesgue measure on $\mathbb{D}$ ($d\mathnormal{A}=\frac{1}{\pi} dxdy=\frac{1}{\pi} r dr d\theta $).
	\end{thm}
	\begin{proof}
		We start by	considering the right hand side of equation of Theorem 3.3., and converting $f$ into polar form, we obtain
		\begin{align*}
			|f(0)|^{2}+2\int_{\mathbb{D}}|f'(z)|^{2}\log \frac{1}{|z|} d\mathnormal{A} =|f(0)|^{2}+2\int_{-\pi}^{\pi} \frac{1}{\pi} \int_{0}^{1}|f'(re^{i\theta})|^{2} (\log \frac{1}{r})r dr d\theta
		\end{align*}
		Interchanging the two integrals which can be justified by Fubini's theorem; we have the following		    
		\begin{align*}
			|f(0)|^{2}+2\int_{\mathbb{D}}|f'(z)|^{2}\log \frac{1}{|z|} d\mathnormal{A}		
			=|f(0)|^{2}+2\int_{0}^{1} (\frac{1}{\pi} \int_{-\pi}^{\pi}|f'(re^{i\theta})|^{2} d\theta) (\log \frac{1}{r})r dr                  
		\end{align*}
		Applying simple algebraic manipulations, we obtain 	  
		\begin{align*}
			|f(0)|^{2}+2\int_{\mathbb{D}}|f'(z)|^{2}\log \frac{1}{|z|} d\mathnormal{A}=|f(0)|^{2}+4\int_{0}^{1} \mathnormal{M}_{2}^{2}(f', r) (\log \frac{1}{r})r dr
		\end{align*}
		Replacing by the Taylor series representation, we get
		\begin{align*}
			|f(0)|^{2}+2\int_{\mathbb{D}}|f'(z)|^{2}\log \frac{1}{|z|} d\mathnormal{A}
			=|f(0)|^{2}+4\int_{0}^{1} \sum_{n=1}^{\infty} (|n|^{2}|a_{n}|^{2}+ |n|^{2}|b_{n}|^{2}) r^{(2n-2)} (\log \frac{1}{r})r dr\\
			=|f(0)|^{2}+4\sum_{n=1}^{\infty} |n|^{2}(|a_{n}|^{2}+ |b_{n}|^{2}) \int_{0}^{1} r^{(2n-1)} (\log \frac{1}{r})r dr	\\
			=|f(0)|^{2}+4\sum_{n=1}^{\infty} n^{2}(|a_{n}|^{2} + |b_{n}|^{2}) \frac{1}{4n^{2}}\\
			= |f(0)|^{2}+\sum_{n=1}^{\infty} (|a_{n}|^{2}+ |b_{n}|^{2})=||f||_{\mathnormal{H}_{h}^{2}}. 	
		\end{align*}
		From which we obtain,
		$$|f(0)|^{2}+2\int_{\mathbb{D}}|f'(z)|^{2}\log \frac{1}{|z|} d\mathnormal{A}=||f||_{\mathnormal{H}_{h}^{2}}. $$	
		This completes the proof.		
		
	\end{proof} 
	
	\section{Growth estimates and Kernels} 
	
	The analogous growth estimates and reproducing kernels on space of complex-valued harmonic functions in the unit disc can be obtained as follows. 
	
	\begin{thm} (Growth estimate). For each $z \in  \mathbb{D}$, the growth estimate of $f \in \mathnormal{H}_{h}^{2}$ is given by
		\begin{center}
			$|f(z)| \leq \frac{{2}||f||_{\mathnormal{H}_{h}^{2}}}{\sqrt{1-|z|^{2}}}.$
		\end{center}
	\end{thm}	
	\begin{proof}
		By applying the triangle inequality for the modulus and Cauchy-Schwartz inequality to the complex-valued harmonic function of $f$ in Theorem 3.4 we obtain for each $z \in \mathbb{D}$, 
		\begin{align*}
			|f(z)|=|\sum_{n=0}^{\infty}a_{n}z^{n}+\sum_{n=0}^{\infty} \overline{b_{n}z^{n}}|  
			\leq \sum_{n=0}^{\infty}|a_{n}||z^{n}| +\sum_{n=0}^{\infty} |{b_{n}||z^{n}}|\\
			\leq (\sum_{n=0}^{\infty}|a_{n}|^{2})^{\frac{1}{2}} (\sum_{n=0}^{\infty}|z^{n}|^{2})^{\frac{1}{2}} +(\sum_{n=0}^{\infty} |b_{n}|^{2})^{\frac{1}{2}} (\sum_{n=0}^{\infty}|z^{n}|^{2})^{\frac{1}{2}}	\\
			\leq [(\sum_{n=0}^{\infty}|a_{n}|^{2})^{\frac{1}{2}} + (\sum_{n=0}^{\infty} |b_{n}|^{2})^{\frac{1}{2}}](\sum_{n=0}^{\infty}|z|^{2n})^{\frac{1}{2}}\\
			\leq (||h||_{\mathnormal{H}_{h}^{2}} +||g||_{\mathnormal{H}_{h}^{2}})(\sum_{n=0}^{\infty}|z|^{2n})^{\frac{1}{2}} \leq \frac{||h||_{\mathnormal{H}_{h}^{2}} +||g||_{\mathnormal{H}_{h}^{2}}}{\sqrt{1-|z|^{2}}}.
		\end{align*}			But then again, $||h||_{\mathnormal{H}_{h}^{2}} \leq ||f||_{\mathnormal{H}_{h}^{2}};  ||g||_{\mathnormal{H}_{h}^{2}} \leq ||f||_{\mathnormal{H}_{h}^{2}}$ implying that, 
		
		$$|f(z)| \leq \frac{2||f||_{\mathnormal{H}_{h}^{2}}}{\sqrt{1-|z|^{2}}}. $$	
		
	\end{proof}	
	
	\begin{thm}(Kernel)
		The reproducing kernel function $K_{\alpha}$ , for a point $\alpha$ in $\mathbb{D}$ is defined by 
		\begin{center}
			$K_{\alpha}(z)=\frac{1}{1-\overline{\alpha}{z}}+ \frac{1}{1-\alpha \overline{z}} \hspace{2mm} $for $  |\overline{\alpha} z|<1. $
		\end{center}
	\end{thm}
	\begin{proof}
		For a point $\alpha \in \mathbb{D}$, $K_{\alpha}$ is the functional in  $ \mathnormal{H}_{h}^{2}(\mathbb{D})$ such that for all $f$ in $ \mathnormal{H}_{h}^{2}(\mathbb{D})$, 	we have 
		\begin{center}
			$\langle f, K_{\alpha} \rangle=f(\alpha)$
		\end{center}
		where $f$ and $K_{\alpha}$ are in $\mathnormal{H}_{h}^{2}$. Let	$f(z)=\sum_{n=0}^{\infty}a_{n}z^{n}+\sum_{n=0}^{\infty} \overline{b_{n}z^{n}}$ \hspace{0.25mm} and \hspace{0.25mm} $K_{\alpha}(z)=\sum_{n=0}^{\infty}c_{n}z^{n}+\sum_{n=0}^{\infty} \overline{d_{n}z^{n}}$  for some coefficients. Thus for each $f \in \mathnormal{H}_{h}^{2}$, 
		\begin{center}
			$\sum_{n=0}^{\infty}a_{n}\alpha^{n}+\sum_{n=0}^{\infty} \overline{b_{n}\alpha^{n}} = f(\alpha) = \langle f, K_{\alpha} \rangle = \sum_{n=0}^{\infty}a_{n}\overline{c_{n}}+\overline{b_{n}}d_{n}$.	
		\end{center}
		This implies that 
		\begin{center}
			$\sum_{n=0}^{\infty}a_{n}\alpha^{n}+\sum_{n=0}^{\infty} \overline{b_{n}} \hspace{1.5mm}\overline{\alpha^{n}} =\sum_{n=0}^{\infty}a_{n}\overline{c_{n}}+ \sum_{n=0}^{\infty}\overline{b_{n}} \hspace{1.5mm}d_{n}.$
		\end{center}
		So this can be true if the following holds true   
		\begin{center}
			$\alpha^{n}=\overline{c_{n}} $  \hspace{12mm} and \hspace{12mm} $\overline{\alpha^{n}}= d_{n} $.
		\end{center} 
		Thus,
		\begin{center}
			$K_{\alpha}(z)=\sum_{n=0}^{\infty}\overline{\alpha^{n}}{z^{n}}+\sum_{n=0}^{\infty}{\alpha^{n}} \overline{z^{n}}$
		\end{center}
		\begin{center}
			$=\sum_{n=0}^{\infty}(\overline{\alpha}z)^{n}+\sum_{n=0}^{\infty} (\alpha\overline{z})^{n}$
		\end{center}
		\begin{center}
			$=\frac{1}{1- \overline{\alpha} z}+ \frac{1}{1-\alpha \overline{z}} $ .
		\end{center}
		Therefore, the required result holds true for $|\overline{\alpha} z|<1.$ 	
	\end{proof}
	
	For a point $\alpha$ in the unit disk $\mathbb{D}$ , because the kernel function $K_{\alpha}$ is a functional in $\mathnormal{H}_{h}^{2}(\mathbb{D})$ , we have 
	\begin{center}
		$||K_{\alpha}||_{\mathnormal{H}_{h}^{2}}^{2}=\langle K_{\alpha}, K_{\alpha} \rangle = K_{\alpha} (\alpha)= \frac{1}{1-\overline{\alpha}{\alpha}}+ \frac{1}{1-\alpha \overline{\alpha}}=2(\frac{1}{1-|\alpha|^{2}}).$
	\end{center}

	\subsection*{Conclusion}
	In general, analogues to the analytic case we defined Hilbert space of complex-valued harmonic functions on the unit disc and obtained, equivalent norm representation in terms of integral means; proved Littlewood-Paley Identity Theorem. Moreover we obtained analogous growth estimate of complex-valued harmonic functions on the unit disc and defined the reproducing kernel which is a functional on the space under consideration.

\end{document}